\documentclass[12pt]{amsart}

\usepackage[all]{xy}
\usepackage{amsfonts}
\usepackage{amssymb}

\usepackage{amsmath,amssymb,amsthm, latexsym, amscd}

\usepackage{color}

\newtheorem{teo}{Theorem}[section]
\newtheorem{lemma}[teo]{Lemma}
\newtheorem{prop}[teo]{Proposition}

\newtheorem{cor}[teo]{Corollary}

\theoremstyle{definition}

\def\<{\langle}
\def\>{\rangle}

\def\a{\alpha}

\def\t{\tau}

\def\N{{\mathbb N}}

\def\Z{{\mathbb Z}}

\def\B{{\mathcal B}}

\def\1{\mathbf 1}

\def\N{{\mathbb N}}

\begin{document}

\title{On Topologies on the Group $(\Z_p)^{\N}$}

\author{I.~K.~Babenko and S.~A.~Bogatyi}

\thanks{This work was partially supported by the Russian
Foundation for Basic Research (project no.~16-01-00357A), and by
ANR-Finsler.}

\address{Universit\'e de Montpellier CNRS UMR 5149,
Institut Montpelli\'erain Alexander Grothen\-dieck, Place Eug\`ene
Bataillon, B\^at. 9, CC051, 34095 Montpellier CEDEX
5, France;\\
 Mechanics and Mathematics Faculty, M.~V.~Lomonosov Moscow State University,
Moscow 119992, Russia} \email{ivan.babenko@umontpellier.fr;
bogatyi@inbox.ru}

\maketitle

\begin{abstract}
It is proved that, on any Abelian group of infinite cardinality
${\bf m}$, there exist precisely $2^{2^{\bf m}}$ nonequivalent
bounded Hausdorff group topologies. Under the continuum
hypothesis, the number of nonequivalent compact and locally
compact Hausdorff group topologies on the group $(\Z_p)^{\N}$ is
determined.
\end{abstract}
\maketitle

The systematic study of Abelian topological groups was initiated
in Pontryagin's paper~\cite{Pontjagin34} and van~Kampen's
paper~\cite{vKampen.35} (see also~\cite{Alexander}). In these
papers, an exceptionally important relationship between an Abelian
topological group and the topological group of its characters was
revealed. It is not an exaggeration to say that this relationship
is one of the most important tools for studying Abelian
topological groups. This relationship manifests itself most
pronouncedly for compact and locally compact groups.

 Topological groups close to compact and locally compact ones are, respectively,
totally bounded and locally bounded Abelian topological groups.
This type of topologies on Abelian groups was studied by A.
Weil~\cite{Weil37}, who showed that an Abelian topological group
is totally bounded if and only if it is embedded in a compact
Abelian group as a dense subgroup and that such a group is locally
bounded if and only if it is embedded in a locally compact group
as a dense subgroup. Another deep characteristic property of
locally bounded and totally bounded topological, not necessarily
Abelian, groups were found by A.D. Alexandroff
\cite{Alexandroff42}. He showed in particular that a topological
group $G$ is locally bounded if and only if $G$ admits a left
invariant generalized measure, see \cite{Alexandroff42} for more
details.

Among all bounded group topologies on a given Abelian group there
is a maximal topology; the Weil completion of a group with this
topology coincides with the Bohr compactification of this group
and can be continuously mapped onto any compact group extension.

The next important step in the study of Abelian topological groups
was made by Kakutani \cite{Kakutani43}, who showed that the
cardinality of an infinite compact Abelian group $G$ and the
cardinality of the discrete group $G^*$ of its characters are
related by
$$
|G| = 2^{|G^*|}.
$$

In the 1950s,  interest in the study of general properties of
Abelian topological groups motivated the interest in the study of
the relationship between the algebraic structure of a given group
and the possible topologies consistent with its group structure.
This topic includes, in particular, the problem of describing
Abelian groups admitting a compact topology, which was posed by
Kaplansky in~\cite{Kaplansky54}. This problem was completely
solved by Hulanicki in~\cite{Hulanicki57, Hulanicki58}.

We mention two results concerning the possible compact topologies
on a given Abelian group. In the very early 1950s, Kulikov
\cite{Kulikov52} showed that the cardinality of the family of all
compact Abelian groups of cardinality $2^{\bf m}$, where ${\bf m}$
is an infinite cardinal, equals $2^{\bf m}$. In the light of this
result the following example of Fuchs~\cite{Fuchs59} looks quite
surprising. He constructed an Abelian group $G$ of cardinality
$2^{\bf m}$ admitting infinitely many different compact
topologies. Moreover, the cardinality of the set of pairwise
non-homeomorphic compact invariant topologies on $G$
equals~$2^{\bf m}$.

It is well known that, on any Abelian group~$G$, there are as many
nonequivalent Hausdorff group topologies as possible. To be more
precise, the cardinality of the set of nonequivalent Hausdorff
group topologies on $G$ equals~$2^{2^{|G|}}$. This result was
obtained by Podewski~\cite{Podewski77}, who studied the problem in
its most general setting. Namely, given any set with a finite or
countable system of algebraic structures, Podewski found necessary
and sufficient conditions on these algebraic structures ensuring
the existence of a nondiscrete Hausdorff topology consistent with
all of them~\cite{Podewski77}. Moreover, Podewski showed that
these conditions ensure the existence of the maximal possible
number of pairwise non-equivalent Hausdorff topologies which agree
with all algebraic structures. For Abelian groups and fields,
Podewski's conditions always hold; details can be found
in~\cite{Podewski77}. In 1970, Arnautov constructed an example of
an infinite commutative ring such that any Hausdorff ring topology
on this ring is discrete~\cite{Arnaut.70}. In the 1979s--1980s,
examples of infinite non-Abelian groups admitting only discrete
Hausdorff group topologies were
given~\cite{Hes79},~\cite{Olsh80},~\cite{Shel80}.

Fuchs' example shows that properties of various topologies that an
abstract Abelian group $G$ carries directly depend on the
algebraic properties of the group~$G$.

In our view, the influence of the algebraic structure of an
Abelian group on the structure of the set of Hausdorff group
topologies on this group has been studied insufficiently. All
topologies considered in the following are assumed to be
Hausdorff.

This paper studies the structure of the set of possible topologies
on the well known Abelian group
\begin{equation}\label{eq:1}
G = (\Z_p)^{\N} \hskip4pt ,
\end{equation}
where $\Z_p = \Z / p\Z$ for a prime $p$. Henceforth for two sets
$X$ and $Y$ the set $X^Y$ is understood in usual sense of set
theory. As customary in topology, we denote the prime field with
$p$ elements by the same symbol~$\Z_p$.

Note that there are four canonical, in a certain sense, topologies
on~$G$. The group $G$ has the natural compact topology determined
by the structure of the direct product~(\ref{eq:1}); we denote
this topology by $\tau_c$. The second topology is the discrete
topology $\tau_d$ on~$G$. The character group of $(G, \tau_c)$,
which we denote by $(G, \tau_c)^*=\oplus_{\aleph_0}\Z_p$, is
countable and discrete. Thus, we obtain two topological groups,
$(G, \tau_c) \times (G, \tau_c)^*$ and $(G, \tau_c) \times (G,
\tau_d)$. Since these two groups are algebraically isomorphic to
$G$, it follows that, in fact, we obtain two more topologies  on
$G$, which are surely different (they have different weights); we
denote these topologies by $\tau_0$ and $\tau_1$, respectively.
Obviously, these two topologies are locally compact. The following
theorem is one of the main results of this paper.

\begin{teo}\label{teo:principal}
For the Abelian group $G = (\Z_p)^{\N}$, the following assertions
hold.

\noindent \textup1. The set of all Hausdorff group topologies on
$G$ contains a subset of cardinality $2^{2^{2^{\aleph_0}}}$
consisting of pairwise non-homeomorphic totally bounded
topologies.

\noindent \textup2. Under the continuum hypothesis, $\tau_c$,
$\tau_d$, $\tau_0$, and $\tau_1$ are the only locally compact
Hausdorff group topology on $G$. Only one of them, $\tau_c$, is
compact.
\end{teo}

As mentioned above, the general result of
Podewski~\cite{Podewski77} affirms that there are
$2^{2^{2^{\aleph_0}}}$ pairwise non-homeomorphic group topologies
on $G$. The first assertion of the theorem shows that the set of
all group topologies on $G$ contains a subset of the same
cardinality consisting of {\it totally bounded} pairwise
non-homeomorphic topologies on $G$.

Although totally bounded topologies are very close to compact,
under the continuum hypothesis, only four of these numerous
topologies are locally compact, and only one of them is compact.
This fact contrasts with Fuchs' result in~\cite{Fuchs59} cited
above.

Theorem~\ref{teo:principal} follows from more general statements
proved below. Thus, its first assertion is, in fact, a general
property of Abelian groups and follows directly from
Theorem~\ref{teo:number.of.topologies} proved in
Section~\ref{sec:all.topologies}.

Compact and locally compact topologies are studied in
Section~\ref{sec:loc.comp.topologies} in a more general situation.
However, in Section~\ref{sec:loc.comp.topologies}, we essentially
use the algebraic structure of~$G$, namely, the fact that each of
its element has order~$p$. The second assertion of the theorem
follows directly from
Corollary~\ref{cor:loc.compact.topology.aleph.1} and, according to
this corollary, is equivalent to the continuum hypothesis. Note
that the negation of the continuum hypothesis drastically changes
the situation with locally compact topologies; see
Proposition~\ref{prop:compact.topology.aleph.1}.

Section~\ref{sec:absolutely.bounded.topologies} considers the set
of all totally bounded topologies on the group~$G$. This subset,
which has the largest possible cardinality in the set of all
topologies on $G$, has a number of interesting structures. For
example, we can single out lattices of topologies. This, in turn,
allows us to construct chains of decreasing totally bounded
topologies on~$G$. Moreover, there exist chains of maximum length,
which have maximal or minimal elements. These structures are
studied in Section~\ref{sec:absolutely.bounded.topologies}.

One of our main tools for studying Abelian topological groups is,
as usual, the character theory founded by Pontryagin
in~\cite{Pontjagin34}. This theory has been  developed most fully
for locally compact Abelian groups, for which Pontryagin's duality
theorem is valid (see \cite{Pont73}). Among other classes of
topological groups whose connection with their character groups
remains strong enough are the already mentioned classes of totally
bounded and locally bounded groups. The relationship between these
groups and their character groups was studied by Comfort and
Ross~\cite{Comfort.Ross}.

For other classes of Abelian topological groups, connection with
character groups is not so pronounced;  nevertheless, character
groups remain an important tool for studying them.

In conclusion, we mention that our interest in this topic was
aroused by a series of questions asked to us by R.~I.~Grigorchuk.
We thank him for numerous fruitful discussions. We are also
grateful to O.~V.~Sipacheva for useful conversations.


\section{The Structure of the Set of All
Topologies}\label{sec:all.topologies}

Let  ${\bf m}$ be a finite or infinite cardinal. In what follows,
we use the notation $\prod_{\bf m}\Z_p =
\mathop{\prod}\limits_{\alpha \in A}(\Z_p)_{\alpha}$, where the
cardinality of the index set is $|A| = {\bf m}$; accordingly,
$\sum_{\bf m}\Z_p = \mathop{\oplus}\limits_{\alpha \in
A}(\Z_p)_{\alpha}$.

Given another cardinal ${\bf k}$, we set
\begin{equation}\label{eq:F_k,m}
F_{\bf k, \bf m}(p) = \prod_{\bf k}\Z_p \times \sum_{\bf m}\Z_p.
\end{equation}
In the case  where $p$ fixed or is clear from the context, we
denote groups of the form~(\ref{eq:F_k,m}) by~$F_{\bf k, \bf m}$.

Endowing the first factor with the compact product topology and
the second with the discrete topology, we turn $F_{\bf k, \bf m}$
into an Abelian topological group. In what follows, we assume that
at least one of the cardinals ${\bf k}$ and ${\bf m}$ is infinite
and the finite cardinals can take only the value~$0$. The
easy-to-prove properties of this Abelian topological group are
collected in the following proposition.
\begin{prop}\label{prop:structure.F_k_m}
The topological group $F_{\bf k, \bf m}$ is locally compact; it
has cardinality \tolerance3000 $\max\{2^{\bf k}, \bf m\}$ and
weight $\max\{{\bf k}, {\bf m}\}$. The group of characters of
$F_{\bf k, \bf m}$ is isomorphic to the group $F_{\bf m, \bf k}$.
For different pairs of cardinals $({\bf k}, {\bf m})$, the
topological spaces $F_{\bf k, \bf m}$ are nonhomeomorphic.
\end{prop}

Below we give general facts about topologies on Abelian groups
generated by subgroups of the character group.

We denote the entire group of characters of an Abelian group $A$,
i.e., the group of all possible homomorphisms of $A$ to the circle
$S^1 =\{e^{it}\}$, by $A^{\star}$. If $A$ is endowed with a group
topology, then we denote the group of topological characters,
i.e., continuous homomorphisms, of the Abelian topological group
$A$ to the circle by $A^{\ast}$. Note that $A^{\ast}$ is always a
subgroup of $A^{\star}$; if the topology of $A$ is discrete, then
these two groups coincide.

As usual, a subset $X \subset A^{\star}$ is said to separate
vectors if, for any different $v, w \in A$, there exists an $h \in
X$ such that $h(v) \neq h(w)$. Any set $X \subset A^{\star}$
generates a diagonal map
\begin{equation}\label{eq:2.Diag}
\Delta_X: A \longrightarrow (S^1)^X.
\end{equation}
A set $X$ separates vectors if and only if the diagonal map
(\ref{eq:2.Diag}) is a monomorphism. Note that $X$ separates
vectors if and only if so does the subgroup $\langle X\rangle$
generated by this set.

As is known \cite{Pont73}, for any nonzero element $v \in A$,
there exists a character $\phi_v \in A^{\star}$ such that
$\phi_v(v) \neq 1$. Uniting such characters over all elements $v
\in A$, we obtain a set $X \subset A^{\star}$ separating the
elements of $A$. Let $H = \langle X\rangle$ be the subgroup
generated by this set; as mentioned above, $H$ separates points as
well. If $|A| = {\bf m}$, where ${\bf m}$ is an infinite cardinal,
then, obviously, we have $|H| \leq {\bf m}$. Note also that the
diagonal map $\Delta_H$ implements an isomorphism of the group $A$
onto its image in $\mathop{\prod}\limits_{h \in H}(S^1)_h$.

Given a subgroup $H \subset A^{\star}$, we denote the minimal
topology on $A$ with respect to which all maps in $H$ are
continuous by $\tau_H$. This topology coincides with the minimal
topology with respect to which  the diagonal map~(\ref{eq:2.Diag})
is continuous (the image is assumed to be endowed with the product
topology). Obviously, $\tau_H$ is Hausdorff if and only if $H$
separates points. Recall that a topology $\tau $ on a group $G$ is
said to be totally bounded if, for any neighborhood $U$ of zero,
there exists a finite set $\{0,g_1,\ldots ,g_k\}$ of elements of
$G$ such that the family $\{U,g_1+U,\ldots ,g_k+U\}$ covers the
entire group.

Topologies determined by various subgroups $H \subset A^{\star}$
were studied by Comfort and Ross~\cite{Comfort.Ross}; we collect
the results of~\cite{Comfort.Ross} needed later on in the
following theorem.

\begin{teo}\cite{Comfort.Ross}\label{teo:Comfort-Ross}
\textup1. For any subgroup $H \subset A^{\star}$, the topology
$\tau_H$ is totally bounded, and for each totally bounded topology
$\tau$ on $A$, there exists a subgroup $H \subset A^{\star}$ such
that $\tau = \tau_H$.

\textup2. The identity map $id : (A, \tau_{H_2}) \longrightarrow
(A, \tau_{H_1})$ is continuous if and only if $H_1 \subset H_2$.
\end{teo}
Now, we prove a general statement on the number of totally bounded
topologies on an Abelian group.

\begin{teo}\label{teo:number.of.topologies}
Let $A$ be an Abelian group of cardinality $|A| = {\bf m}$, where
${\bf m}$ is an infinite cardinal. Then, on $A$, there are
precisely $2^{2^{\bf m}}$ totally bounded Hausdorff group
topologies. Among them there are $2^{2^{\bf m}}$ pairwise
nonhomeomorphic topologies.
\end{teo}
\begin{proof}
Let $A^{\star}$ be the group of all characters of $A$; by
Kakutani's theorem, we have $|A^{\star}| = 2^{\bf m}$. Fix a
subgroup $H \lhd A^{\star}$ separating the points of characters.
As mentioned above, this subgroup can always be chosen so that
$|H| \leq \bf m$. Let $B = A^{\star}/H$; then $|B| = 2^{\bf m}$.

Given a subgroup $C \lhd B$, we set $\bar{C} = p^{-1}(C)$, where
$p:A \longrightarrow B$ is the natural projection. Since $H\lhd
\bar{C}$, it follows that the topology $\tau_{\bar{C}}$ on $A$ is
Hausdorff. Thus, to each subgroup of $B$ we have assigned a
totally bounded Hausdorff group topology on $A$. These topologies
are different for different subgroups. Indeed, according to
Theorem~\ref{teo:Comfort-Ross}, the identity map is not a
homeomorphism with respect to topologies corresponding to
different subgroups. In turn, Lemma~\ref{lemma:subgroup.number}
implies that the cardinality of the set of such topologies is
$2^{2^{\bf m}}$.

Finally, let us count  nonequivalent topologies among those
obtained above. By nonequivalent topologies we mean topologies
that cannot be mapped to each other by a homeomorphism, not
necessarily consistent with the group structure.

Since $|A| = {\bf m}$, it follows that the set of all self-maps of
$A$ has cardinality ${\bf m}^{\bf m} = 2^{\bf m}$ (as a set). It
follows that any topology on $A$ is equivalent to at most $2^{\bf
m}$ other topologies on this set. Thus, we obtain a partition of
$2^{2^{\bf m}}$ different totally bounded group topologies into
classes of equivalent topologies, each of which has cardinality at
most $2^{\bf m}$. Therefore, the number of classes is precisely
$2^{2^{\bf m}}$, and topologies from different classes are
pairwise nonequivalent. This completes the proof of the theorem.
\end{proof}
\begin{lemma}\label{lemma:subgroup.number}
Let $\bf m$ be an infinite cardinal, and let $B$ be an Abelian
group of cardinality $2^{\bf m}$. Then the cardinality of the set
of pairwise different subgroups in $B$ equals $2^{2^{\bf m}}$.
\end{lemma}
\begin{proof}
The upper bound for the cardinality of the set of subgroups in $B$
is obvious and coincides with the cardinality of the set of all
subsets in $B$. Let us prove the lower bound.

Let $T = \text{Tors}(B)$ be the torsion subgroup of $B$; then $B'
= B/ T$ is torsion-free. The cardinality of at least one of the
groups $T$ and $B'$ equals $2^{\bf m}$, and the presence of the
required number of pairwise different subgroups in $T$ or in $B'$
implies the presence of the
 same number of subgroups in $B$. Thus, it suffices to prove the lemma
 for periodic groups and torsion-free groups separately.
Consider two cases.

1. Suppose that the group $B$ is periodic. Since $B$ decomposes
into the direct sum of its primary components \cite{KM72}, it
follows that at least one primary component has cardinality
$2^{\bf m}$, and it suffices to prove the required assertion for
the case of a primary $p$-group, where $p$ is a prime.

Suppose that $B$ is a $p$-group. Consider the subgroup $B(p)$
consisting of  all elements of order $p$. Obviously, $|B(p)| =
2^{\bf m}$, so that it suffices to prove the required assertion
for the group $B(p)$. In this case, $B(p)$ is a $\Z_p$-vector
space, and it has a Hamel basis of cardinality $2^{\bf m}$.
Different subsets of this basis generate different subgroups of
$B(p)$ and, therefore, of $B$. This completes the proof in the
case of periodic groups.

2. Finally, suppose that $B$ is torsion-free. Since $|B| = 2^{\bf
m}$, it follows that the
 (Pr\"ufer) rank of $B$ equals $2^{\bf m}$. Thus,  there is a free $\Z$-module of rank $2^{\bf
m}$ in $B$. As above, different subsets in a family of generators
of such a module generate different subgroups in the module and,
therefore,  in $B$. \end{proof}

\subsection{The Structure of the Family of Totally Bounded
Topologies}\label{sec:absolutely.bounded.topologies}

In what follows, we essentially use the structure of the groups
$F_{\bf k, \bf m}(p)$ and, in particular, the fact that these
group can be treated as $\Z_p$-vector spaces. The characters of
such groups are simply the elements of the dual space. In what
follows, we use this dual understanding of characters without
mention.

Consider the case where the group $A$ is a $\Z_p$-vector space of
dimension equal to an infinite cardinal ${\bf m}$. In this case,
Theorem~\ref{teo:number.of.topologies} is substantially simpler
and has a transparent combinatorial interpretation. The subgroup
$H \lhd A^{\star}$ of characters separating points is a subspace
of the same dimension ${\bf m}$. Next, we choose a Hamel basis
$\B$ in $H$ and extend it to a Hamel basis $\B^{\star}$ in the
entire space $A^{\star}$. Let $M = \B^{\star} \setminus \B$;
obviously, we have $|M| = 2^{\bf m}$.

Each set $N \subset M$ now determines a subgroup in $A$ being the
linear span of the system of vectors $\{N \cup \B\}$. The set of
 pairwise different subgroups thus defined has cardinality
$2^{2^{\bf m}}$ and can be used in
Theorem~\ref{teo:number.of.topologies} for the special case where
$A$ is a $\Z_p$-vector space. The Hausdorff group topologies on
$A$ thus obtained have an additional structure: on this set of
topologies, there arise lattices.

The construction described above shows that, for the set $M =
\B^{\star} \setminus \B$ of cardinality $2^{\bf m}$, the lattice
$(2^M; \subset, \cup, \cap)$ of subsets is embedded in the lattice
of totally bounded linearly invariant topologies on the space~$A$.
Therefore, the structures on the lattice $(2^M; \subset, \cup,
\cap)$ can be transferred to topologies. Thus, any linearly
ordered set $M$ generates a chain of length $2^{\bf m}$ of
one-to-one continuous maps. In the family of topologies under
consideration, there is the minimal topology $\tau_{\text{min}} =
\tau_{H_{\B}}$ and the maximal topology $\tau_{\text{max}} =
\tau_{A^{\star}}$.

Thus, on a vector space $A$ of dimension ${\bf m}$, there are
chains of  length $2^{\bf m}$ of one-to-one continuous maps, and
there are no longer chains. Such chains carry various
combinatorial structures, which correspond to different types of
ordering of the vectors in the Hamel basis.

Moreover, the existence of $2^{2^{\bf m}}$ independent subsets in
$M$ implies the existence of $2^{2^{\bf m}}$ totally bounded
linearly invariant Hausdorff topologies on $A$ such that the
identity map is not continuous  with respect to any pair of these
topologies.

\begin{prop}\label{teo:separated.min.subspace}
Let $A$ be a vector space of infinite dimension ${\bf m}$. Then
any separating subspace $H \subset V^{\star}$ contains a
separating subspace of dimension at most~${\bf m}$.
\end{prop}
\begin{proof}
Indeed, given any pair of different vectors in the space $A$, it
suffices to take a  linear functional on $H$ separating them. The
dimension of the linear span of such functionals does not exceed
${\bf m}^2 = {\bf m}$. \end{proof}

\begin{prop}
Let $\tau$ be a totally bounded linearly invariant Hausdorff
topology on a vector space $A$ of infinite dimension ${\bf m}$.
Then there is a set $M$ of cardinality $2^{\bf m}$ such that, on
the set of totally bounded linearly invariant Hausdorff topologies
on $A$, there is a lattice isomorphic to $(2^M; \subset, \cup,
\cap)$. Moreover, the initial topology $\tau$ is either minimal or
maximal in this lattice. In particular, any totally bounded
linearly invariant Hausdorff topology on $A$ can be included in a
linear chain of length $2^{\bf m}$ of different topologies.
\end{prop}

\begin{proof}
Consider the space $H$ of linear functionals generating the given
topology. If the dimension of $H$ is less than $2^{\bf m}$, then
we take the Hamel basis $\B$ in $H$, extend it to a Hamel basis
$\B^{\star}$ in $A^{\star}$, and set $M = \B^{\star} \setminus
\tilde{\B}$. In this case, $(2^M; \subset, \cup, \cap)$ is the
required lattice, and the given topology is minimal in this
lattice.

Suppose that the dimension of $H$ equals $2^{\bf m}$. According to
Proposition~\ref{teo:separated.min.subspace}, $H$ contains a
separating subspace $Q \subset H$ of dimension $\bf m$. We take a
Hamel basis $\B$ in  $Q$ and extend it to a Hamel basis
$\tilde{\B}$ in $H$. As above, we set $M = \tilde{\B} \setminus
\B$. The desired lattice is again $(2^M; \subset, \cup, \cap)$,
and the topology determined by $H$ is  maximal in it.
\end{proof}

\begin{prop}
Let ${\bf k}$ be an infinite cardinal, and let
$A=F_{k,0}=\prod_{\bf k}\Z_p$. Then the subspace $H_0 \subset
A^{\star}$ generated by the coordinate projections of $A$ is a
minimal separating space.
\end{prop}

\begin{proof}
Obviously, the subspace $H_0$ itself is separating. Suppose that
there exists a separating subspace $H \subset H_0$. Consider the
one-to-one continuous map generated by the embeddings $(A,
\tau_{H_0}) \longrightarrow (A, \tau_H)$. The topology of
$\tau_{H_0}$ coincides with the Tychonoff product topology on $A$
and, therefore, is compact. Since a continuous one-to-one map of a
compact space to any Hausdorff space is a homeomorphism, the
separating subspace $H$ cannot be proper.
\end{proof}


The following result shows that there are no other minimal
separating subspaces.

\begin{teo}\label{teo:separating.subspaces}
For a $\Z_p$-vector space $A$, a subspace $H \subset A^{\star}$ is
a minimal separating space if and only if the natural map ${\bf
e}: A \longrightarrow H^{\star}$ is an isomorphism.
\end{teo}

\begin{proof}
Obviously, the natural map ${\bf e}$ is a monomorphism if and only
if $H \subset A^{\star}$ is a separating subspace.

Suppose that ${\bf e}$ is an isomorphism and $H_1 \subset H$ is a
proper subspace. Since the map $p: H^{\star} \longrightarrow
H_1^{\star}$ of the dual spaces  has nontrivial kernel, it follows
that so does the natural map ${\bf e}_1 = p \circ {\bf e}$ in
$H_1^{\star}$. Therefore, $H_1$ does not separate points, i.e.,
$H$ is minimal.

Now, let us suppose that $H$ is a minimal separating space and
show that ${\bf e}$ is an isomorphism. Take a nonzero element
$\phi \in H^{\star}$ and let $\ker \phi = H_1 \subset H$. By
minimality, $H_1$ is not separating; hence
 there exists a nonzero vector $v \in A$ such that
${\bf e}(v)(x) = x(v) =0$ for all $x \in H_1$. Since $H$ separates
points, it follows that there exists a $y \in H$ for which $y(v)
\neq 0$. Without loss of generality, we can assume  that $y(v) =
1$.

The subspace  $H_1$ has codimension 1, because this is the kernel
of a linear functional. Therefore, $H_1$ and $y$ generate $H$; in
particular, $\phi(y)\neq 0$. Finally, for any vector $z \in H$, we
have $z = x + ty$, where $x \in H_1$. This implies
$$
\phi(z) = t\phi(y) = t\phi(y)y(v) = ty(\phi(y)v) =  (x +
ty)(\phi(y)v) = {\bf e}(\phi(y)v)(z).
$$
Thus, $\phi = {\bf e}(\phi(y)v)$, so that ${\bf e}$ is an
epimorphism and, by the remark at the beginning of the proof, an
isomorphism.
\end{proof}

\begin{cor}
Let $A$ be a $\Z_p$-vector space for which $H \subset A^{\star}$
is a minimal separating subspace of infinite dimension ${\bf k}$.
Then there exists an isomorphism $i: \prod_{\bf k}\Z_p
\longrightarrow A$ such that $i^{\star}(H) \subset (\prod_{\bf
k}\Z_p)^{\star}$ coincides with the subspace generated by the
basis projections of $\prod_{\bf k}\Z_p$.
\end{cor}

\begin{proof}
Indeed, it suffices to choose a Hamel basis in $H$; this
determines an isomorphism $H \simeq \sum_{\bf k}\Z_p$, whence $A
\simeq H^{\star} \simeq \prod_{\bf k}\Z_p$.
\end{proof}

We define the {\it logarithm} of a cardinal $\bf m$ by
$$
\text{Ln}({\bf m}) = \{{\bf n}: {\bf m} = 2^{\bf n}\}.
$$
Obviously,  ${\bf k} \in \text{Ln}(2^{\bf k})$ for any cardinal
${\bf k}$. Moreover, the set $\text{Ln}({\aleph_0})$ is empty by
Cantor's theorem. Finally, the generalized continuum hypothesis
($2^{\aleph_{\a}} = \aleph_{\a+1}$) implies $\text{Ln}(2^{\bf k})
= \{{\bf k}\}$.

\begin{cor}
If the dimension $\bf m$ of a vector space $A$ satisfies the
condition $Ln({\bf m}) = \emptyset$, then $A^{\star}$ contains no
minimal separating subspace.
\end{cor}

\section{Locally Compact Topologies}\label{sec:loc.comp.topologies}

In this section, we analyze the set of locally compact topologies
arising on the group $\prod_{\bf n}\Z_p$, where $\bf n$ is an
infinite cardinal.

Let $G$ be an infinite compact Abelian group; then, according to
Kakutani's theorem \cite{Kakutani43}, we have $|G| = 2^{\bf m}$,
where $m=|G^{\ast}|$. Since the character group $G^{\ast}$ is
discrete, it follows that if any element of $G$ has order $p$,
then the group $G$ is isomorphic to one of the groups $G \simeq
\prod_{\bf n}\Z_p$, where ${\bf n} \in \text{Ln}(|G|)$, with the
natural product topology. For different ${\bf n} \in
\text{Ln}(|G|)$, the topological groups $\prod_{\bf n}\Z_p$ are
nonhomeomorphic, because, according to Proposition
\ref{prop:structure.F_k_m}, the weight of $\prod_{\bf n}\Z_p$
equals ${\bf n}$. Finally, note that under the generalized
continuum hypothesis there exists only one compact topology on
each direct product $\prod_{\bf n}\Z_p$. Thus, the following
assertion is valid.
\begin{prop}\label{prop:compact.topology}
If ${\bf n}$ is an infinite cardinal, then the cardinality of the
set of compact Hausdorff group topologies on the group $G \simeq
\prod_{\bf n}\Z_p$ equals $|\text{Ln}(|G|)|$. Moreover, each ${\bf
n}' \in \text{Ln}(|G|)$ corresponds to precisely one compact
topology, and its weight equals ${\bf n}'$.
\end{prop}
\begin{teo}\label{teo:loc.compact.topologies}
Let $\t$ be a locally compact Hausdorff topology on the group $G
\simeq \sum_{\bf m}\Z_p$. Then the group $(G, \t)$ is
topologically isomorphic to $F_{\bf k, \bf n}$, where the
cardinals $\bf k$ and $\bf n$ satisfy the condition $\max \{2^{\bf
k}, {\bf n}\} = {\bf m}$.
\end{teo}
\begin{proof} All elements of the group $G$ and, therefore, all elements
of its character group $G^{\ast}$ have order  $p$. It follows that
the group $(G, \t)$ is totally disconnected.

According to van~Dantzig's theorem \cite{Dantzig.32}, there exists
a compact open subgroup $H \subset G$. Let $p: G \longrightarrow N
= G/H$ be the natural projection, and let $i: H \longrightarrow G$
be an embedding. Since $H$ is open, it follows that $N$ is
discrete and, therefore, $N \simeq \sum_{\bf n}\Z_p$, where $\bf
n$ is a cardinal. According to
Proposition~\ref{prop:compact.topology}, we have $H \simeq
\prod_{\bf k}\Z_p$ for an appropriate cardinal $\bf k$ (naturally,
the direct product is endowed with the compact product topology).

Next, choosing a Hamel basis in the $\Z_p$-vector space $N$, we
take a splitting $q: N \longrightarrow G$ of the projection $p$.
The homomorphism $q$ is automatically continuous, and its image is
discrete in $G$.

Finally, consider the homomorphism $i{\times}q: H{\times}N
\longrightarrow G$ defined by $i{\times}q(h, n) = i(h) + q(n)$.
Obviously, $i \times q$ is an algebraic isomorphism. Since $H$ is
open, it follows that this is a  local homeomorphism between
$H{\times}N$ and $G$, which, in turn, implies that $i{\times}q$ is
an isomorphism of topological groups.

Thus, we conclude that $(G, \t)$ is isomorphic to $F_{\bf k, \bf
n}$ as a topological group. The equality $\max \{2^{\bf k}, {\bf
n}\} = {\bf m}$ is now obvious.
\end{proof}

The following proposition shows what happens to locally compact
topologies on the group $\sum_{\bf n}\Z_p$, where ${\bf n} =
\aleph_1$ or $2^{\aleph_0}$, under the negation of the continuum
hypothesis.
\begin{prop}\label{prop:compact.topology.aleph.1}
If $\aleph_1 < 2^{\aleph_0}$, then there is only one (discrete)
locally compact topology on the group $\sum_{\aleph_1}\Z_p$. Under
the same assumption, there exist at least five pairwise
nonhomeomorphic locally compact topologies on the group
$\sum_{2^{\aleph_0}}\Z_p$.
\end{prop}
\begin{proof}
According to Theorem~\ref{teo:loc.compact.topologies}, the group
$\sum_{\aleph_1}\Z_p$ with a locally compact topology is
isomorphic to a group $F_{\bf k, \bf n}$. Taking into account the
inequality $\aleph_1 < 2^{\aleph_0}$, we see that the only
possible choice of the cardinals is ${\bf k}=0$ and ${\bf n} =
\aleph_1$, i.e., the topology is discrete.

The groups $F_{\bf k, \bf n}$ are locally compact and, under the
assumption $\aleph_1 < 2^{\aleph_0}$, algebraically isomorphic to
$\sum_{2^{\aleph_0}}\Z_p$ for the following pairs  $({\bf k, \bf
n})$ of cardinals:
$$
\hskip5pt (2^{\aleph_0}, 0),\hskip5pt (2^{\aleph_0}, \aleph_0),
\hskip5pt (2^{\aleph_0}, \aleph_1), \hskip5pt (2^{\aleph_0},
\aleph_2), \hskip5pt (0, 2^{\aleph_0}).
$$
Here the first topology is compact and the last one is discrete.
All of these topologies are pairwise nonhomeomorphic by virtue of
Proposition~\ref{prop:structure.F_k_m}.
\end{proof}

Theorem~\ref{teo:loc.compact.topologies} and Proposition
\ref{prop:compact.topology.aleph.1} have the following immediate
corollary.

\begin{cor}\label{cor:loc.compact.topology.aleph.1}
The following conditions are equivalent:

\noindent \textup{(1)} the continuum hypothesis holds: $\aleph_1 =
2^{\aleph_0}$;

\noindent \textup{(2)} the group $\sum_{\aleph_1}\Z_p$ admits a
compact group topology;

\noindent \textup{(3)} the group $\sum_{\aleph_1}\Z_p$ admits a
nondiscrete locally compact group topology;

\noindent \textup{(4)} the group $\prod_{\aleph_0}\Z_p$ admits at
most four locally compact group topologies;

\noindent \textup{(5)} the group $\prod_{\aleph_0}\Z_p$ admits
precisely four locally compact group topologies.
\end{cor}
In conclusion, we mention that the four locally compact
topological groups in condition (5) are precisely as follows:

$F_{\aleph_0, 0}$ (this is a compact topology of weight
$\aleph_0$);

$F_{\aleph_0, \aleph_0}$ (a noncompact nondiscrete topology of
weight $\aleph_0$);

$F_{\aleph_0, 2^{\aleph_0}}$ (a noncompact nondiscrete topology of
weight $2^{\aleph_0}$);

$F_{0, 2^{\aleph_0}}$ (the discrete topology of weight
$2^{\aleph_0}$).

This list completes the proof of the second assertion of
Theorem~\ref{teo:principal} stated in the introduction.

\bigskip

{
}

\end{document}